\documentclass[reqno]{amsart}
\usepackage{a4wide,cite}
\usepackage{amsmath,amssymb,mathrsfs}
\usepackage{xcolor}
\definecolor{revisionred}{RGB}{0,0,0}
\newif\ifshowrevisions
\showrevisionsfalse
\newenvironment{revision}
  {\begingroup\ifshowrevisions\color{revisionred}\fi}
  {\endgroup}
\newcommand{\rev}[1]{{\ifshowrevisions\color{revisionred}\fi#1}}
\newcommand{\pdfrev}[1]{{\color{black}#1}}
\usepackage{hyperref}

\hypersetup{
  colorlinks=true,
  allcolors=black,
  pdftitle={The UMD property of symmetric operator spaces}
}
\numberwithin{equation}{section}
\newtheorem{theorem}{Theorem}[section]
\newtheorem{lemma}[theorem]{Lemma}
\newtheorem{proposition}[theorem]{Proposition}
\newtheorem{corollary}[theorem]{Corollary}
\theoremstyle{remark}

\newcommand{\M}{\mathcal M}
\newcommand{\tr}{\tau}
\newcommand{\Z}{\mathbb Z}
\newcommand{\ind}{\mathbf1}
\newcommand{\norm}[1]{\left\|#1\right\|}

\newcommand{\dd}{\,\mathrm d}

\newcommand{\E}{\mathbb E}
\newcommand{\Prob}{\mathbb P}
\newcommand{\cF}{\mathcal F}
\newcommand{\rsupp}{\operatorname{r}}

\hypersetup{linkcolor=blue}

\title[UMD property of symmetric operator spaces]{The UMD property of symmetric operator spaces}

\begin{document}
\begin{abstract}
We prove that if $E$ is a UMD symmetric Banach function space on
$(0,\infty)$, then $E(\M,\tau)$ is UMD for every semifinite von Neumann
algebra $\M$ equipped with a faithful normal semifinite trace $\tau$.
This resolves in the affirmative an open problem that has circulated in the non-commutative world for more than four decades.
\end{abstract}

\author{J. Huang}
\address{Institute for Advanced Study in Mathematics of HIT,
Harbin Institute of Technology, Harbin 150001, China}
\email{jinghao.huang@hit.edu.cn}

\author{F. Sukochev}
\address{School of Mathematics and Statistics,
University of New South Wales, Sydney NSW 2052, Australia}
\email{f.sukochev@unsw.edu.au}

\thanks{This work was supported by the
  NNSF of China (Nos. 12301160, 12471134 and 12671159) and the  ARC  (DP230100434).}
\maketitle

\section{Introduction}\label{sec:introduction}

The passage from a symmetric Banach function space $E$ to the
corresponding space $E(\M,\tau)$ of measurable operators
associated with a semifinite von Neumann algebra $\M$ equipped
with a faithful normal semifinite trace $\tau$ is a basic
construction in noncommutative integration\cite{DPS,KS,v37}. A natural question is which
geometric properties are preserved by this construction. In the present
paper we consider the UMD property, that is, the unconditionality of
martingale differences.

The question whether the UMD property of $E$ implies that of
$E(\M,\tau)$
or of the symmetrically normed ideal $C_E$ {\color{black}(in the case when $\M$ coincides with the algebras of all bounded linear operators on infinite-dimensional Hilbert space)}
has circulated for more than four decades, {\color{black}and 
several results   have been established under
the additional assumption that $E(\mathcal M,\tau)$ is UMD; see, for
example, the results on RUC-decompositions in
\cite[Proposition~2.1]{DS1997} and on Poincar\'e type
inequalities in
\cite[Theorems~1.1(4) and~6.1]{BenEfraimLustPiquard2008}.}
The second-named author learned about this question from Arazy \cite{Arazy} who investigated this problem
during the 1980s and 1990s.
Lust-Piquard and Pisier\cite{LP} stated this problem in the context of operator ideals:
\begin{quote}
     It is apparently not known whether 
``$E$ UMD'' implies ``$C_E$ UMD''.
\end{quote}
The following (semifinite) version of this question is quoted from \cite{DS1997} (see also \cite[Section~4, final paragraph]{Randrianantoanina2002}):
\begin{quote}
    On the other hand, it is not known
if $E(\M, \tau)$ has the UMD-property whenever $E$ is a UMD space, even in the case of the
unitary matrix spaces.
\end{quote}
It was subsequently
recorded by Pisier and Xu in the following form \cite[p.~1494]{PX}.
\begin{quote}
\textsc{Problem 7.8.}
Let $\M$ be a semifinite von Neumann algebra equipped with an n.s.f.\
trace $\tau$, and let $E$ be a UMD r.i.\ space on $(0,\infty)$.
Is $E(\M,\tau)$ a UMD space?
\end{quote}
Here \emph{n.s.f.} means normal, semifinite and faithful, and
\emph{r.i.} means rearrangement-invariant. In the paragraph preceding
the problem, Pisier and Xu noted that it had already circulated for
almost two decades. \pdfrev{We note that the so-called  analytic UMD property does not transfer from the commutative counterpart $E$ to its noncommutative counterpart $E(\M,\tau)$\cite{HP}. }

The first positive results concern the noncommutative $L^p$-spaces.
For $1<p<\infty$, their UMD property follows from the work of Bourgain
\cite{Bourgain} and Berkson, Gillespie and Muhly \cite{BGM}; see also
\cite[Corollary~7.7]{PX}. Real interpolation gives further examples.
Cobos \cite{Cobos} proved preservation of UMD under the real
interpolation method
and applied this result to classes of
Lorentz--Marcinkiewicz operator ideals. 

For Orlicz spaces, a positive result was obtained by Dodds and
Sukochev \cite[Corollary~1.8]{DS} (see also the historical account
in \cite{LVY}): $L^\Phi(\M,\tau)$ has the UMD-property whenever $L^\Phi$ is a reflexive Orlicz space. 
Bekjan's Orlicz martingale inequalities \cite{Bekjan} yield
another route to the conclusion for Orlicz spaces, as noted in
\cite[Remark~3.2(i)]{BC}. Bekjan and Chen
\cite[Corollary~3.2 and Remark~3.2(ii)]{BC} subsequently obtained the
UMD property of noncommutative Orlicz spaces whose lower and upper
Matuszewska--Orlicz indices lie strictly between $1$ and $\infty$.
Their argument uses interpolation of noncommutative martingale
inequalities. 

Bukhvalov \cite{Bukhvalov} studied the corresponding commutative
vector-valued theory. His Proposition~1.3 states that, for a nonzero
Banach ideal space $E$ and a nonzero Banach space $Y$, the
K\"othe--Bochner space $E(Y)$ is UMD if and only if both $E$ and $Y$
are UMD. He also proved that, if $E$ is a symmetric space on a
nonatomic probability space and interpolates between $L^{p_0}$ and
$L^{p_1}$, where $1<p_0<p_1<\infty$, uniform boundedness of martingale
transforms in the $E(Y)$-norm is equivalent to the UMD property of $Y$
\cite[Lemma~2.1]{Bukhvalov}. The first of these results is a
inheritance  principle used in the present proof.

Our main result concerns an arbitrary UMD symmetric Banach function
space $E$ and its noncommutative counterpart $E(\M,\tau)$;
see \cite[Chapter~4]{DPS} and \cite{KS,LSZ} for this construction.
For $1<u<\infty$ and a Banach space $Y$, we denote by
$\beta_{u,Y}$ the least constant in the $L^u(\Omega;Y)$
martingale transform inequality (see \eqref{eq:UMD} below).
The precise definitions are recalled in
Section~\ref{sec:preliminaries}. 

\begin{theorem}\label{thm:main}
Let $E$ be a UMD symmetric Banach function space on $(0,\infty)$.
For every semifinite von Neumann algebra $\M$ equipped with a faithful
normal semifinite trace $\tau$, the space $E(\M,\tau)$ is also UMD.
Moreover, for every $1<u<\infty$ there is a constant $C_{E,u}$ such that
\[
 \beta_{u,E(\M,\tau)}\le C_{E,u}
\]
for all such $(\M,\tau)$.
\end{theorem}

Below, we explain our strategy for the proof of Theorem~\ref{thm:main} by showing the $E(\M,\tau)$ is isomorphic to a quotient of a UMD space.
Using the type and cotype of
$E$, we choose $1<p<q<\infty$ as in \eqref{eq:exponents}.
Put $I_n=[2^n,2^{n+1})$ for $n\in\Z$.
For a sequence $z=(z_n)_{n\in\Z}$ with
$z_n\in L^p(\M,\tau)\cap L^q(\M,\tau)$, define the
nonnegative scalar functions
\[
   f_\ell(t)=\sum_{n\in\Z}2^{-n/\ell}
                \norm{z_n}_\ell\ind_{I_n}(t),
   \qquad t>0,\quad \ell\in\{p,q\},
\]
where $\norm{\cdot}_\ell$ is the norm on $L^\ell(\M,\tau)$.
We define $Z$ to consist of the sequences for which
$f_p,f_q\in E$, with norm
\[
   \norm{z}_Z=\max\{\norm{f_p}_E,\norm{f_q}_E\}.
\]
This is the scalar form of the norm in \eqref{eq:Znorm}.
Lemma~\ref{lem:Zumd} identifies $Z$ isometrically with a
closed subspace of the Banach space
\[
 E\bigl(L^p(\M,\tau)\bigr)\oplus_\infty
 E\bigl(L^q(\M,\tau)\bigr).
\]
Here $E(Y)$ denotes the K\"othe--Bochner space defined in
Section~\ref{sec:preliminaries}. By
\cite[Proposition~1.3]{Bukhvalov}, the two
K\"othe--Bochner spaces are UMD. Consequently $Z$ is UMD,
by stability under finite direct sums and closed subspaces;
the latter property is stated in
\cite[Proposition~1.2(1)]{Bukhvalov}.

For every $z\in Z$, the series $\sum_{n\in\Z}z_n$
converges absolutely in $L^p(\M,\tau)+L^q(\M,\tau)$.
Indeed, estimates
\eqref{eq:absolute-head}--\eqref{eq:absolute-tail} give
\[
   \sum_{n\in\Z}
      \norm{z_n}_{L^p(\M,\tau)+L^q(\M,\tau)}
   \le \sum_{n\le0}\norm{z_n}_p
      +\sum_{n>0}\norm{z_n}_q<\infty.
\]
Consequently, the series defines a bounded linear operator
\begingroup
\[
   Q:Z\longrightarrow L^p(\M,\tau)+L^q(\M,\tau),\qquad
   Qz:=\sum_{n\in\Z}z_n.
\]
\endgroup
We next show that $Q$ takes values in $E(\M,\tau)$ and is
bounded as an operator from $Z$ to $E(\M,\tau)$. For $k\in\Z$, we
split the series and estimate $\sum_{n\le k}z_n$ in $L^p(\M,\tau)$
and $\sum_{n>k}z_n$ in $L^q(\M,\tau)$.
Write $\mu_t$ for the generalized singular value function
defined in Section~\ref{sec:preliminaries}.
Applying the two estimates in \eqref{eq:sv-add}
(see \cite[Lemma~2.5(v) and Remark~3.3]{FK}) gives
\[
 \mu_{2^{k+1}}(Qz)
 \le 2^{-k/p}\sum_{n\le k}\norm{z_n}_p
    +2^{-k/q}\sum_{n>k}\norm{z_n}_q.
\]
Let $\mathscr H_p$ and $\mathscr T_q$ be the positive
operators defined in \eqref{eq:hardy}. For $t\in I_k$, the
right-hand side is exactly
$(\mathscr H_pf_p)(t)+(\mathscr T_qf_q)(t)$.
Thus monotonicity of $\mu(Qz)$ yields
\[
   \mu_{2t}(Qz)\le
   (\mathscr H_pf_p)(t)+(\mathscr T_qf_q)(t),
   \qquad\text{for almost every }t>0.
\]
The bounds $A,B$ for these two operators are given in
\eqref{eq:Hpbound}--\eqref{eq:Tqbound}.
Taking the $E$-norm, and applying the dilation
$(D_2f)(t)=f(t/2)$, we obtain
\begin{align*}
   \norm{Qz}_{E(\M,\tau)}
   &\le \norm{D_2}_{E\to E}
          \bigl(A\norm{f_p}_E+B\norm{f_q}_E\bigr)\\
   &\le C_Q\norm{z}_Z,
   \qquad C_Q:=\norm{D_2}_{E\to E}(A+B)<\infty.
\end{align*}
Therefore $Q:Z\longrightarrow E(\M,\tau)$ is a bounded linear
operator; the details are given in Proposition~\ref{prop:Q}.

To prove that the summation map $Q:Z\to E(\M,\tau)$ is
onto, fix $x\in E(\M,\tau)$. We construct $z\in Z$ with
$Qz=x$ using the spectral projections
$\ind_{(\lambda,\infty)}(|x|)$, $\lambda\ge0$.
Their traces are controlled by the singular values through
\[
   \tau\bigl(\ind_{(\mu_t(x),\infty)}(|x|)\bigr)\le t,
   \qquad t>0;
\]
see \eqref{eq:strict-support}.
Proposition~\ref{prop:lift} uses functional calculus to
construct bounded operators $z_n$ such that
$z=(z_n)\in Z$, $Qz=x$ and
$\norm{z}_Z\le L\norm{x}_{E(\M,\tau)}$, where $L$ depends
only on $E$.
Thus $Q$ induces a Banach space isomorphism
\[
  \widetilde Q:Z/\ker Q\longrightarrow E(\M,\tau),
  \qquad \widetilde Q(z+\ker Q)=Qz,
\]
with $\norm{\widetilde Q}\le C_Q$ and
$\norm{\widetilde Q^{-1}}\le L$.
The UMD property of $E(\M,\tau)$ now follows from
inheritance  under quotients
\cite[Proposition~1.2(1)]{Bukhvalov} and Banach space
isomorphisms. In Section~\ref{sec:mainproof} we give the
corresponding quantitative martingale transform estimate.

Section~\ref{sec:preliminaries} contains the notation and the facts
used in the proof. The quotient representation is established in
Section~\ref{sec:representation}. The proof of
Theorem~\ref{thm:main} is completed in Section~\ref{sec:mainproof}.
\rev{Section~\ref{sec:extensions} establishes the corresponding
results for symmetric sequence spaces and for symmetric function
spaces on $(0,1)$.}

\section{Preliminaries}\label{sec:preliminaries}

\subsection{Symmetric function spaces}
All Banach spaces in the present paper are over the complex field. Scalar functions on
$(0,\infty)$ are identified modulo equality almost everywhere with
respect to Lebesgue measure. For a measurable function $f$, we denote
by $f^*$ the decreasing rearrangement of $|f|$.
For background on symmetric function spaces and their noncommutative
counterparts, we refer to the monographs \cite{KPS} and \cite[Chapter~4]{DPS}.

A symmetric Banach function space $E$ on $(0,\infty)$ is a nonzero
Banach space of measurable functions with the following property:
if $g\in E$ and $f^*\le g^*$, then $f\in E$ and
$\norm{f}_E\le\norm{g}_E$. We use the usual Banach function-space
convention that indicators of sets of finite measure belong to $E$
and integration over such sets defines a continuous seminorm.
The associate space $E'$ consists of measurable functions $g$ for which
\[
 \norm{g}_{E'}=
 \sup_{\norm{f}_E\le1}\int_0^\infty |f(t)g(t)|\dd t<\infty.
\]
In particular,
$\int|fg|\le\norm{f}_E\norm{g}_{E'}$.

 We do not assume the Fatou property in this definition. Recall that
 this property means that $0\le f_n\uparrow f$ and
 $\sup_n\norm{f_n}_E<\infty$ imply
 $f\in E$ and $\norm{f}_E=\sup_n\norm{f_n}_E$.
A reflexive Banach function space has the Fatou property and an
order-continuous norm. These facts follow from the associate-space
duality described in \cite[Chapter~I, \S2, p.~8, and Chapter~II,
Introduction, pp.~44--45]{KPS}.
Consequently, both properties hold whenever $E$ is UMD.

\subsection{Measurable operators}
Let $\M\subset B(H)$ be a von Neumann algebra with identity $\ind$,
equipped with a faithful normal semifinite trace $\tau$.
We write $S(\M,\tau)$ for the algebra of $\tau$-measurable operators
affiliated with $\M$. Thus an element $x$ is a closed densely defined
affiliated operator such that, for every $\delta>0$, there exists a
projection $e\in\M$ with $\tau(\ind-e)<\delta$,
$eH\subset\operatorname{dom}(x)$ and $xe\in\M$.
Algebraic operations in $S(\M,\tau)$ are understood by taking closures.

The measure topology on $S(\M,\tau)$ is the Hausdorff topology with
a basis of neighbourhoods of zero given by
\[
 V(\varepsilon,\delta)=
 \{x:\text{ there is a projection }e\in\M,\ 
       \tau(\ind-e)<\delta,\ \norm{xe}_\infty<\varepsilon\},
 \qquad\varepsilon,\delta>0.
\]
Here $\norm{\cdot}_\infty$ is the operator norm on $\M$.
For $x\in S(\M,\tau)$ let $|x|=(x^*x)^{1/2}$, and denote the spectral
projection of $|x|$ corresponding to a Borel set $B$ by
$\ind_B(|x|)$. Its right support is
$\rsupp(x)=\ind_{(0,\infty)}(|x|)$. If $x=v|x|$ is the polar
decomposition, then $v^*v=\rsupp(x)$.

The distribution function and the generalized singular value
function are defined by
\[
 d_x(a)=\tau\bigl(\ind_{(a,\infty)}(|x|)\bigr),\qquad
 \mu_t(x)=\inf\{a\ge0:d_x(a)\le t\}\quad(t>0),
\]
where $\inf\varnothing=\infty$. We write $\mu(x)$ for the
right-continuous decreasing function $t\mapsto\mu_t(x)$.
For the theory of generalized singular value functions, see
\cite{FK} and \cite[Chapter~3]{DPS}.
When $\tau(\ind)<\infty$, this function is zero for
$t\ge\tau(\ind)$. A sequence $x_n$ converges to zero in measure if
and only if $\mu_t(x_n)\to0$ for every $t>0$.

For $1\le a<\infty$,
\[
 L^a(\M,\tau)=\{x\in S(\M,\tau):\tau(|x|^a)<\infty\},
 \qquad \norm{x}_a=\tau(|x|^a)^{1/a},
\]
and $L^\infty(\M,\tau)=\M$. The symmetric operator space associated
with $E$ is
\[
 E(\M,\tau)=\{x\in S(\M,\tau):\mu(x)\in E\},
 \qquad \norm{x}_{E(\M,\tau)}=\norm{\mu(x)}_E.
\]
This is a Banach space with the displayed norm by
\cite[Theorems~8.7 and 8.11]{KS}; see also \cite{DDD} for the
earlier construction and \cite[Chapter~4]{DPS} for a systematic treatment of
symmetric operator spaces. We occasionally write $E(\M)$ for
$E(\M,\tau)$.

We shall use the following singular value estimates
\cite[Lemma~2.5 and Corollary~2.8]{FK}:
\begin{equation}\label{eq:sv-add}
 \mu_{s+t}(x+y)\le\mu_s(x)+\mu_t(y),\qquad
 \mu_t(x)\le t^{-1/a}\norm{x}_a\quad(x\in L^a(\M,\tau)).
\end{equation}
Positive parts and expressions such as $\min(|x|,\lambda)$ are
defined by spectral calculus from the corresponding scalar functions.

\subsection{Bochner spaces and the UMD property}
Let $(\Omega,\cF,\Prob)$ be a probability space and let $Y$ be a
Banach space. A $Y$-valued function is strongly measurable if it is
the almost everywhere pointwise limit of finite-valued measurable
simple functions; see \cite[Section~1.1]{HNVW1}.
For $1\le u<\infty$, the Bochner space
$L^u(\Omega;Y)$ consists of equivalence classes of strongly
measurable functions $F:\Omega\to Y$ for which
\[
 \norm{F}_{L^u(\Omega;Y)}
 =\left(\int_\Omega\norm{F(\omega)}_Y^u\dd\Prob(\omega)\right)^{1/u}
 <\infty.
\]
When the probability space is fixed, we also write $L^u(Y)$ for
$L^u(\Omega;Y)$. For a sub-sigma-algebra $\mathcal G\subset\cF$, the conditional expectation
$\E(\,\cdot\,\mid\mathcal G)$ is the usual Bochner conditional
expectation. It is a contraction on these spaces; see \cite[Section~2.6]{HNVW1}.

For a finite filtration
$\cF_0\subset\cF_1\subset\cdots\subset\cF_m$, a sequence
$(d_j)_{j=1}^m$ in $L^u(\Omega;Y)$ is a martingale difference
sequence if $d_j$ is $\cF_j$-measurable and
$\E(d_j\mid\cF_{j-1})=0$. The UMD constant $\beta_{u,Y}$ is the
least constant such that
\begin{equation}\label{eq:UMD}
 \left\|\sum_{j=1}^m\varepsilon_jd_j\right\|_{L^u(\Omega;Y)}
 \le\beta_{u,Y}
 \left\|\sum_{j=1}^m d_j\right\|_{L^u(\Omega;Y)}
\end{equation}
for all probability spaces, finite filtrations, martingale difference
sequences, and deterministic signs $\varepsilon_j\in\{-1,1\}$.
The space $Y$ is UMD if $\beta_{u,Y}<\infty$ for one, equivalently
every, $u\in(1,\infty)$.

The K\"othe--Bochner space $E(Y)$ consists of equivalence classes
of strongly measurable functions $F:(0,\infty)\to Y$ such that
$\norm{F(\cdot)}_Y\in E$, with norm
\[
 \norm{F}_{E(Y)}=\bigl\|\norm{F(\cdot)}_Y\bigr\|_E.
\]
The underlying measure is Lebesgue measure. In particular,
$E(L^a(\M,\tau))$ denotes a space of operator-valued functions on
$(0,\infty)$ with this norm.

We recall the facts needed below. UMD spaces are reflexive and have
nontrivial Rademacher type and finite Rademacher cotype; see
\cite{HNVW1,HNVW2}. The UMD property passes to closed subspaces
\cite[Proposition~1.2(1), p.~752]{Bukhvalov} and finite direct sums
\cite[Section~3.1, p.~43]{Amann2009}.
Moreover, for $a,u\in(1,\infty)$,
\begin{equation}\label{eq:known-UMD}
 \beta_{u,L^a(\M,\tau)}\le b_{a,u},\qquad
 \beta_{u,E(Y)}\le c_{E,u}\beta_{u,Y},
\end{equation}
where $b_{a,u}<\infty$ is a constant that depends on $a$ and $u$ only, and $c_{E,u}<\infty$ if $E$ is UMD, which depends on $E$ and $u$ only.
The first assertion is the noncommutative $L^p$ UMD theorem
\cite{Bourgain,BGM,PX}. The second is the K\"othe--Bochner theorem
of Rubio de Francia \cite{Rubio}; see also
\cite[Proposition~1.3]{Bukhvalov}. The stated quantitative form follows
from \cite[Theorems~6.6.3--6.6.4]{Lorist}. Reflexivity of $E$
ensures the function-space hypotheses of that result.

For a compatible Banach couple $(U,V)$, the algebraic sum
$U+V$ is equipped with the norm
\[
  \norm{x}_{U+V}
  =\inf\{\norm{y}_U+\norm{z}_V:
                y\in U,\ z\in V,\ x=y+z\};
\]
see \cite[Chapter~I, \S3, p.~9, formula~(3.2)]{KPS} or \cite{DPS}.

The   direct sum $U\oplus_\infty V$ is equipped with the  norm
$\norm{(y,z)}=\max\{\norm{y}_U,\norm{z}_V\}$.

\section{A quotient representation}\label{sec:representation}
Throughout this section, we always assume that  $E$ is a UMD symmetric Banach function space
on $(0,\infty)$.
\subsection{Dilation estimates}

We use the following consequence of the geometry of UMD spaces; see
\cite[Example~7.4.8 and Proposition~7.4.12]{HNVW2}.
\begin{theorem} \label{thm:type-cotype}
Every UMD Banach space has Rademacher type $r$ for some $1<r\le2$
and Rademacher cotype $s$ for some $2\le s<\infty$.\footnote{In  \cite{Bourgain83}, 
one has constructed examples of symmetric function spaces  on $(0,1)$ which satisfy the upper $p$-estimate and the
lower $q$-estimate for $1 < p < q < \infty $, but are not (UMD)-spaces. }
\end{theorem}
 
Fix a Rademacher type exponent $r$ and a Rademacher cotype
exponent $s$ for $E$, with $1<r\le2\le s<\infty$, as in
Theorem~\ref{thm:type-cotype}. Let $T_r,C_s<\infty$ denote the
corresponding type and cotype constants. Thus, for every finite
family $f_1,\ldots,f_N\in E$ and independent
Rademacher variables $\rho_1,\ldots,\rho_N$, each taking the values
$\pm 1$  with equal probability, we have
\begin{align*}
 \left(\mathbb E\norm{\sum_{j=1}^N\rho_jf_j}_E^r\right)^{1/r}
 &\le T_r\left(\sum_{j=1}^N\norm{f_j}_E^r\right)^{1/r},\\
 \left(\sum_{j=1}^N\norm{f_j}_E^s\right)^{1/s}
 &\le C_s\left(\mathbb E\norm{\sum_{j=1}^N\rho_jf_j}_E^s\right)^{1/s}.
\end{align*}
For $a>0$, the dilation operator is defined by
$(D_af)(t)=f(t/a)$, $t>0$.

\begin{lemma}\label{lem:dilations}
For every integer $j\ge0$,
\[
 \norm{D_{2^j}}_{E\to E}\le T_r2^{j/r},\qquad
 \norm{D_{2^{-j}}}_{E\to E}\le C_s2^{-j/s}.
\]
\end{lemma}

Let $\varphi(t)=\norm{\ind_{(0,t)}}_E$ for $t>0$, and put
$c_0=\varphi(1)>0$. Applying Lemma~\ref{lem:dilations} to indicator
functions yields
\begin{equation}\label{eq:fundamental}
 \varphi(2^n)\ge
 \begin{cases}
 c_0T_r^{-1}2^{n/r},&n\le0,\\
 c_0C_s^{-1}2^{n/s},&n\ge0.
 \end{cases}
\end{equation}
In particular, $\varphi(t)\to\infty$ as $t\to\infty$.
Let $h\in E$ be nonnegative and decreasing, and fix $t>0$.
Then
\[
 0\le h(t)\ind_{(0,t)}(v)\le h(v)
 \qquad\text{for almost every }v>0.
\]
The ideal property of $E$ therefore gives
$h(t)\varphi(t)\le\norm{h}_E$.
It follows that $h(t)\to0$ as $t\to\infty$.
In particular, $\mu_t(x)\to0$ as $t\to\infty$ for every
$x\in E(\M,\tau)$.

We now choose and fix real numbers $p,q$ such that
\begin{equation}\label{eq:exponents}
 1<p<r\le2\le s<q<\infty.
\end{equation}
For every $f\in E$, define
\begin{equation}\label{eq:hardy}
 \mathscr H_pf=\sum_{j=0}^\infty 2^{-j/p}D_{2^j}f,
 \qquad
 \mathscr T_qf=\sum_{j=1}^\infty 2^{j/q}D_{2^{-j}}f.
\end{equation}
Lemma~\ref{lem:dilations} shows that both series converge in operator
norm on $E$, and that
\begin{align}
 \norm{\mathscr H_p}_{E\to E}
 &\le A:=T_r\sum_{j=0}^\infty2^{-j(1/p-1/r)}<\infty,
 \label{eq:Hpbound}\\
 \norm{\mathscr T_q}_{E\to E}
 &\le B:=C_s\sum_{j=1}^\infty2^{-j(1/s-1/q)}<\infty.
 \label{eq:Tqbound}
\end{align}
For each $0\le f\in E$, the two series in \eqref{eq:hardy}
converge almost everywhere to $\mathscr H_p f$ and
$\mathscr T_q f$, respectively.

\subsection{The representation space $E(X_p)\oplus_\infty E(X_q)$}

Let $p,q$ be as in \eqref{eq:exponents}, and write
$X_p=L^p(\M,\tau)$ and $X_q=L^q(\M,\tau)$.
These spaces form a compatible Banach couple in the space
$S(\M,\tau)$ of $\tau$-measurable operators, equipped with the
measure topology; see \cite[Chapter 7]{DPS}.
For each $n\in\Z$, put $I_n=[2^n,2^{n+1})$, so that $|I_n|=2^n$.
Define $Z$ to be the vector space of sequences
$z=(z_n)_{n\in\Z}$ with $z_n\in X_p\cap X_q$ for which
\begin{align*}
 F_p(z)(\cdot)&=\sum_{n\in\Z}2^{-n/p}z_n\ind_{I_n}(\cdot)
       \quad\hbox{belongs to }E(X_p),\\
 F_q(z)(\cdot)&=\sum_{n\in\Z}2^{-n/q}z_n\ind_{I_n}(t\cdot)
       \quad\hbox{belongs to }E(X_q).
\end{align*}
Equip $Z$ with a norm\footnote{Homogeneity and the triangle inequality follow from
the linearity of $F_p,F_q$ and the corresponding properties of the
two K\"othe--Bochner norms; If $\norm{z}_Z=0$, then
$F_p(z)=0$ almost everywhere; since $I_n$ has positive measure,
$z_n=0$ for every $n\in\Z$.}
\begin{equation}\label{eq:Znorm}
 \norm{z}_Z=\max\left\{\norm{F_p(z)}_{E(X_p)},
                         \norm{F_q(z)}_{E(X_q)} \right\}.
\end{equation}
Both functions $F_p(z),F_q(z)$ are strongly measurable. Indeed,
for either $\ell=p$ or $\ell=q$, the finite-valued measurable functions
\[
 F_{\ell,N}(z)(t)
 =\sum_{n=-N}^{N}2^{-n/\ell}z_n\ind_{I_n}(t)
\]
converge in $X_\ell$ to $F_\ell(z)(t)$ for every $t>0$,
i.e., $F_\ell(z)$ is  strong measurable 
in the sense of  \cite[Section~1.1]{HNVW1}. 

\begin{lemma}\label{lem:Zumd}
$Z$ is a Banach space and the map
$z\mapsto(F_p(z),F_q(z))$ is an isometric closed embedding into
$E(X_p)\oplus_\infty E(X_q)$. Consequently, $Z$ is UMD.
\end{lemma}
\begin{proof}
Since $\norm{\cdot}_Z$ is a norm
 (see \eqref{eq:Znorm}), it suffices to prove that it is complete.

We first prove that the range of $z\mapsto(F_p(z),F_q(z))$ is closed.
For a Banach space $Y$ and $n\in\Z$, define
\[
 P_nF
 =\left(2^{-n}\int_{I_n}F(v)\dd v\right)\ind_{I_n},
 \qquad F\in E(Y).
\]
By the definition of the K\"othe--Bochner norm and
the homogeneity of the norm in $Y$,
\[
  \norm{P_nF}_{E(Y)}
  =2^{-n}\norm{\ind_{I_n}}_E
    \norm{\int_{I_n}F(v)\,\dd v}_Y.
\]
Since $E$ is symmetric and $|I_n|=2^n$, we have
$\norm{\ind_{I_n}}_E=\varphi(2^n)$.
We have 
\[
\begin{aligned}
  \norm{P_nF}_{E(Y)}
\stackrel{\tiny \mbox{triangle~inequality}}{\le} 2^{-n}\varphi(2^n)
       \int_{I_n}\norm{F(v)}_Y\,\dd v \stackrel{\tiny \mbox{H\"older's~inequality}}{\le} 2^{-n}\varphi(2^n)
       \norm{\ind_{I_n}}_{E'}\norm{F}_{E(Y)}.
\end{aligned}
\]
Thus $P_n$ is bounded. 

Suppose that
$(F_p(z^{(k)}),F_q(z^{(k)}))\to(G_p,G_q)$ in the direct sum.
For each $n\in\Z$ and $\ell\in\{p,q\}$ we have
$P_nF_\ell(z^{(k)})=\ind_{I_n}F_\ell(z^{(k)})$.
Passing to the limit in $E(X_\ell)$ yields
$P_nG_\ell=\ind_{I_n}G_\ell$.
Consequently $G_\ell$ is constant almost everywhere on $I_n$,
taking the value
\[
 c_{\ell,n}=2^{-n}\int_{I_n}G_\ell(v)\dd v\in X_\ell.
\]

  Since
\[
F_\ell(z^{(k)})(v)=2^{-n/\ell}z_n^{(k)}
\qquad (v\in I_n)
\]
and $|I_n|=2^n$, it follows that 
\[
2^{-n/\ell}z_n^{(k)}-c_{\ell,n}
=
2^{-n}\int_{I_n}
\bigl(F_\ell(z^{(k)})(v)-G_\ell(v)\bigr)\,\mathrm dv.
\]
Consequently,
\begin{eqnarray*}
\left\|2^{-n/\ell}z_n^{(k)}-c_{\ell,n}\right\|_{X_\ell}
&\stackrel{\tiny \mbox{triangle~inequality}}{\le}& 2^{-n}\int_{I_n}
\left\|F_\ell(z^{(k)})(v)-G_\ell(v)\right\|_{X_\ell}
\,\mathrm dv\\
&\stackrel{\tiny \mbox{H\"older's~inequality}}{\le}& 2^{-n}\|\mathbf1_{I_n}\|_{E'}
\left\|F_\ell(z^{(k)})-G_\ell\right\|_{E(X_\ell)}
\longrightarrow 0
\end{eqnarray*}
as $k \to \infty $. 
In other words,  $z_n^{(k)}$ converges in $X_p$ to $2^{n/p}c_{p,n}$
and in $X_q$ to $2^{n/q}c_{q,n}$.
The embeddings of $X_p$ and $X_q$ into $X_p+X_q$ are continuous,
so uniqueness of the limit in $X_p+X_q$ gives
\[
 z_n:=2^{n/p}c_{p,n}=2^{n/q}c_{q,n}\in X_p\cap X_q.
\]
Let $z=(z_n)_{n\in\Z}$, $z_n \in X_p\cap X_q $. 
Since the intervals  $I_n$ form a countable partition of $(0,\infty)$, it follows from the definitions of $F_p $ and $F_q$ that 
$G_p=F_p(z)$ and $G_q=F_q(z)$ almost everywhere.
Thus $z\in Z$ with $\norm{z}_Z=\max \left\{\norm{G_p}_{E(X_p)} , \norm{G_q}_{E(X_q)} \right\}$.
Since $(G_p,G_q)=(F_p(z),F_q(z))$, the image of the
isometric map $z\mapsto(F_p(z),F_q(z))$ is closed. 
Since $Z$ is isometric to this image, $Z$ is
also complete.

By \eqref{eq:known-UMD}, both $E(X_p)$ and $E(X_q)$ are UMD.
Their finite direct sum is UMD \cite[Section~3.1, p.~43]{Amann2009}.
Since the UMD property passes to closed subspaces
\cite[Proposition~1.2(1), p.~752]{Bukhvalov},
it follows that $Z$ is UMD.
\end{proof}

\subsection{The summation map $Q:Z\longrightarrow X_p+X_q$}\label{sub  Q}

For $z=(z_n)_{n\in \mathbb{Z}}\in Z$ let
\begin{align}\label{an bn}
 a_n=2^{-n/p}\norm{z_n}_p,\qquad
 b_n=2^{-n/q}\norm{z_n}_q,
\end{align}
and put $f_p=\sum_{n\in\Z}a_n\ind_{I_n}$ and
$f_q=\sum_{n\in\Z}b_n\ind_{I_n}$. Then
\[
 \norm{f_p}_E=\norm{F_p(z)}_{E(X_p)},\qquad
 \norm{f_q}_E=\norm{F_q(z)}_{E(X_q)}.
\]
We have 
\[
 a_n\varphi(2^n)\le\norm{f_p}_E,
 \qquad b_n\varphi(2^n)\le\norm{f_q}_E.
\]
Recall from \eqref{eq:exponents} that $p<r$ and $s<q$.
Combining this with \eqref{eq:fundamental} yields
\begin{align}
 \sum_{n\le0}\norm{z_n}_p
 &\le\frac{T_r}{c_0}\norm{f_p}_E
       \sum_{n\le0}2^{n(1/p-1/r)}<\infty,
 \label{eq:absolute-head}\\
 \sum_{n>0}\norm{z_n}_q
 &\le\frac{C_s}{c_0}\norm{f_q}_E
       \sum_{n>0}2^{n(1/q-1/s)}<\infty.
 \label{eq:absolute-tail}
\end{align}
Since $\norm{z_n}_{X_p+X_q}\le
\min(\norm{z_n}_p,\norm{z_n}_q)$, the series is absolutely
convergent in $X_p+X_q$. We can define the bounded linear map
\[
 Q:Z\longrightarrow X_p+X_q,\qquad Qz=\sum_{n\in\Z}z_n.
\]
{\color{black}
Eventually, we show that the quotient mapping $Z/\ker( Q)$ determines a Banach space isomorphism with $E(\M,\tau)$ thus establishing that the latter is a UMD-space. We start working towards this aim.}
\begin{proposition}\label{prop:Q}
The range of $Q$ lies in $E(\M,\tr)$, and
\begin{equation}\label{eq:Qbound}
 \norm{Qz}_{E(\M)}
 \le\norm{D_2}_{E\to E}(A+B)\norm{z}_Z.
\end{equation}
Here $A$ and $B$ are the constants defined in
\eqref{eq:Hpbound} and \eqref{eq:Tqbound}, respectively.
\end{proposition}
\begin{proof}
Fix $t>0$, and let $k\in\Z$ be the unique integer satisfying
$2^k\le t<2^{k+1}$. 
Since each $z_n$ belongs to $X_p\cap X_q$,  it follows from 
\eqref{eq:absolute-head}--\eqref{eq:absolute-tail}  that 
\[
 \sum_{n\le k}\|z_n\|_{X_p}<\infty,
 \qquad
 \sum_{n>k}\|z_n\|_{X_q}<\infty.
\]
We use the following 
decomposition 
\[
 Qz=y_k+w_k,\qquad
 y_k=\sum_{n\le k}z_n\in X_p,\qquad
 w_k=\sum_{n>k}z_n\in X_q.
\]
By \eqref{eq:sv-add}, we have 
\begin{align*}
 \mu_{2^{k+1}}(Qz)
 &\le \mu_{2^k}(y_k)+\mu_{2^k}(w_k)\\
 &\le 2^{-k/p}\sum_{n\le k}\norm{z_n}_p
       +2^{-k/q}\sum_{n>k}\norm{z_n}_q\\
 &\stackrel{\eqref{an bn}}{=}\sum_{j\ge0}2^{-j/p}a_{k-j}
       +\sum_{j\ge1}2^{j/q}b_{k+j}\\
       & \stackrel{\eqref{eq:hardy}}{=}(\mathscr H_pf_p)(t)+(\mathscr T_qf_q)(t),
\end{align*}
where the last equality uses  
$t/2^j\in I_{k-j}$ and $2^jt\in I_{k+j}$.
Since $2t\ge2^{k+1}$ and the singular-value function is decreasing, it follows that 
\begin{equation}\label{eq:pointwiseQ}
 \mu_{2t}(Qz)\le(\mathscr H_pf_p)(t)+(\mathscr T_qf_q)(t) \in E 
 \quad\hbox{for almost every }t>0.
\end{equation}
Applying $D_2$ and using
\eqref{eq:Hpbound}--\eqref{eq:Tqbound}, we obtain
\[
 \norm{Qz}_{E(\M)}
 \le\norm{D_2}_{E\to E}
       \bigl(A\norm{f_p}_E+B\norm{f_q}_E\bigr),
\]
which completes the proof.
\end{proof}

\subsection{The preimage of $Q$ with respect to $E(\M,\tau)$}

We shall use the following consequence of the definition of $\mu_t(x)$:
for $x\in S(\M,\tau)$ and $t>0$ with $\mu_t(x)<\infty$,
\begin{equation}\label{eq:strict-support}
 \tau\bigl(\ind_{(\mu_t(x),\infty)}(|x|)\bigr)\le t.
\end{equation}
{\color{black}To see \eqref{eq:strict-support}}, put $\alpha_j=\mu_t(x)+j^{-1}$ for $j\ge1$.
The definition of 
$d_x$ implies
$d_x(\alpha_j)\le t$.
The projections $\ind_{(\alpha_j,\infty)}(|x|)$ increase to
$\ind_{(\mu_t(x),\infty)}(|x|)$, so normality of $\tau$ yields
\eqref{eq:strict-support}.

\begin{proposition}\label{prop:lift}
Let $p,q$ be the exponents fixed in \eqref{eq:exponents}, and let
$Z$ and $Q$ be as in Lemma~\ref{lem:Zumd} and
Proposition~\ref{prop:Q}. {\color{black}For every $x\in E(\M,\tau)$, there exists
$u=(u_n)_{n\in\Z}\in Z$ such that
\[
 Qu=x,\qquad
 \norm{u}_Z\le 2^{1/p}\norm{D_2}_{E\to E}\norm{x}_{E(\M,\tau)}.
\]
In particular, $Q$ is surjective, and, for every $z\in Z$,
\begin{equation}\label{eq:distance-kernel}
 \operatorname{dist}_Z(z,\ker Q)
 :=\inf_{w\in\ker Q}\norm{z-w}_Z
 \le 2^{1/p}\norm{D_2}_{E\to E}\norm{Qz}_{E(\M,\tau)}.
\end{equation}}
\end{proposition}
\begin{proof}
Fix $x\in E(\M,\tau)$, write $x=v|x|$, and set
$\lambda_n=\mu_{2^n}(x)$ for $n\in\Z$.
The sequence $(\lambda_n)$ is
nonincreasing and tends to zero as $n\to\infty$ (see 
\eqref{eq:fundamental} and the observation following it).
For each $n$, define the bounded nonnegative Borel function
\[
 g_n(t)=\min\bigl((t-\lambda_{n+1})_+,\lambda_n-\lambda_{n+1}\bigr),
 \qquad t\ge0.
\]
In particular, 
$b_n=g_n(|x|)$ 
is a bounded positive operator
in $\M$. Put
\begin{equation}\label{eq:layers}
 u_n=vb_n,\qquad
 e_n=\ind_{(\lambda_{n+1},\infty)}(|x|).
\end{equation}
Since $g_n$ vanishes on $[0,\lambda_{n+1}]$ and
$0\le g_n\le\lambda_n$, we have
\[
 0\le b_n\le\lambda_n e_n,\qquad
 e_nb_n=b_n,\qquad
 e_n\le\ind_{(0,\infty)}(|x|)=v^*v.
\]
Consequently,
$u_n^*u_n=b_nv^*vb_n=b_n^2$, and hence $|u_n|=b_n$.
Moreover, \eqref{eq:strict-support} implies
$\tau(e_n)\le 2^{n+1}$.
For each exponent $\ell\in\{p,q\}$ this yields
\begin{equation}\label{eq:layer-bound}
 \norm{u_n}_\ell^\ell
 =\tau(b_n^\ell)
 \le\lambda_n^\ell\tau(e_n)
 \le 2^{n+1}\lambda_n^\ell.
\end{equation}
In particular, $u_n\in X_p\cap X_q$. Define $u:=(u_n)_{n\in \mathbb{Z}}$. 

Define $h(t)=\sum_{n\in\Z}\lambda_n\ind_{I_n}(t)$.
For every $t>0$, choose $n$ with $t\in I_n=[2^n,2^{n+1})$.
Since $t/2\le 2^n\le t$, monotonicity gives
\[
 \mu_t(x)\le h(t)=\lambda_n\le\mu_{t/2}(x).
\]
Thus $h\in E$ and
$\norm{h}_E\le\norm{D_2}_{E\to E}\norm{x}_{E(\M,\tau)}$.
For $\ell=p,q$, estimate \eqref{eq:layer-bound} gives, on $I_n$,
\[
 \norm{F_\ell(u)(t)}_\ell
 =2^{-n/\ell}\norm{u_n}_\ell
 \le 2^{1/\ell}h(t).
\]
The definition of $\norm{\cdot}_Z$ therefore implies $u\in Z$ and
\begin{equation}\label{eq:liftbound}
 \norm{u}_Z
 \le 2^{1/p}\norm{D_2}_{E\to E}\norm{x}_{E(\M,\tau)}.
\end{equation}
Here $\max(2^{1/p},2^{1/q})=2^{1/p}$ because $p<q$.

It remains to verify $Qu=x$.
The identity
$g_n(t)=(t-\lambda_{n+1})_+-(t-\lambda_n)_+$
implies that for $N\ge1$,
\begin{equation}\label{eq:telescoping}
 \sum_{n=-N}^{N}u_n
 =v\bigl((|x|-\lambda_{N+1})_+-(|x|-\lambda_{-N})_+\bigr).
\end{equation}
Subtracting this identity from $x=v|x|$ gives
\[
 x-\sum_{n=-N}^{N}u_n
 =v\min(|x|,\lambda_{N+1})+v(|x|-\lambda_{-N})_+.
\]
The first term $v\min(|x|,\lambda_{N+1})$ has operator norm at most $\lambda_{N+1}$.
The second term has right support contained in
$\ind_{(\lambda_{-N},\infty)}(|x|)$, whose trace is at most $2^{-N}$
by \eqref{eq:strict-support}.
Fix $t>0$. When $2^{-N}\le t$, we have 
$\mu_t(v(|x|-\lambda_{-N})_+) =0$.
By  \eqref{eq:sv-add}, we have 
\[
 \mu_{2t}\left(x-\sum_{n=-N}^{N}u_n\right)
 \le\lambda_{N+1}\longrightarrow0
\]
as $N\to \infty $. 
It follows that the finite sums $\sum_{n=-N}^{N}u_n$ converge to $x$ in measure.

Since $u\in Z$ and $Qu =\sum_{n \in \mathbb{Z}} u_n  \in X_p+X_q$ (see \eqref{sub  Q}), it follows that $\sum_{n=-N}^{N}u_n$ converges to $Qu$ in measure\cite[Proposition 4.4.4]{DPS}.
Since the measure topology is Hausdorff, it follows that 
\begin{equation}\label{eq:onto}
 Qu=x.
\end{equation}
Thus $Q$ is surjective. 
{\color{black}To prove \eqref{eq:distance-kernel}, fix
$z\in Z$ and apply the preceding construction to $x=Qz$.
The resulting $u\in Z$ satisfies $Qu=Qz$, so $z-u\in\ker Q$.
Consequently, \eqref{eq:liftbound} gives
\[
 \operatorname{dist}_Z(z,\ker Q)
 \le\norm{z-(z-u)}_Z=\norm{u}_Z
 \le 2^{1/p}\norm{D_2}_{E\to E}\norm{Qz}_{E(\M,\tau)}.
\]
This completes the proof.}
\end{proof}

It follows from Propositions~\ref{prop:Q} and \ref{prop:lift} that
the induced map
\[
 \widetilde Q:Z/\ker Q\longrightarrow E(\M,\tau),\qquad
 \widetilde Q(z+\ker Q)=Qz,
\]
is a Banach space isomorphism. {\color{black}Since the quotient norm is
\[
 \norm{z+\ker Q}_{Z/\ker Q}
 =\operatorname{dist}_Z(z,\ker Q),
\]
estimates \eqref{eq:Qbound} and \eqref{eq:distance-kernel} yield
\begin{equation}\label{eq:quotient-bounds}
 \begin{aligned}
  \norm{\widetilde Q}_{Z/\ker Q\to E(\M,\tau)}
  &\le\norm{D_2}_{E\to E}(A+B),\\
  \norm{\widetilde Q^{-1}}_{E(\M,\tau)\to Z/\ker Q}
  &\le 2^{1/p}\norm{D_2}_{E\to E}.
 \end{aligned}
\end{equation}}

\section{Proof of the main theorem}\label{sec:mainproof}

We use the following standard quantitative form of the stability
of UMD under quotients \cite[Proposition~1.2(1)]{Bukhvalov}.
The stated constant follows from the duality identity for UMD
constants \cite[Proposition~8.8, formula~(8.17), p.~171]{PisierIHP}
and the norm estimates for the induced isomorphism $Z/\ker Q\to X$.

\begin{lemma}\label{lem:transfer}
Let $Z$ be UMD, let $X$ be a Banach space, and let $Q:Z\to X$ be
bounded and linear. Suppose that $\norm{Q}\le C_Q$ and that, for
every $x\in X$, there exists $z\in Z$ such that
\[
 Qz=x,\qquad \norm{z}_Z\le L\norm{x}_X.
\]
Then, for every $1<u<\infty$,
\[
 \beta_{u,X}\le C_Q L\,\beta_{u,Z}.
\]
\end{lemma}

The UMD property is stable under finite direct sums; see
\cite[Section~3.1, p.~43]{Amann2009}. 
We record the quantitative estimate needed below for completeness.

\begin{lemma}\label{lem:direct-sum}
Let $U$ and $V$ be UMD Banach spaces, and equip $U\oplus_\infty V$
with the norm
\[
 \norm{(a,b)}_{U\oplus_\infty V}
 =\max\{\norm{a}_U,\norm{b}_V\}.
\]
Then, for every $1<u<\infty$,
\begin{equation}\label{eq:sumconstant}
 \beta_{u,U\oplus_\infty V}
 \le\bigl(\beta_{u,U}^{\,u}+\beta_{u,V}^{\,u}\bigr)^{1/u}.
\end{equation}
\end{lemma}

\begin{proof}
Fix $1<u<\infty$. Let $d_j=(a_j,b_j)$ be a finite
$U\oplus_\infty V$-valued martingale difference sequence, and let
$\varepsilon_j\in\{-1,1\}$.
The bounded coordinate projections commute with conditional
expectation (see \cite[Section~2.6]{HNVW1}). Thus $(a_j)$
and $(b_j)$ are martingale difference sequences with values in
$U$ and $V$, respectively. By the definition of the maximum norm, we have 
\begin{align*}
 \left\|\sum_j\varepsilon_jd_j\right\|_{L^u(U\oplus_\infty V)}^u
 &\le
 \left\|\sum_j\varepsilon_ja_j\right\|_{L^u(U)}^u+
 \left\|\sum_j\varepsilon_jb_j\right\|_{L^u(V)}^u\\
 &\le
 \beta_{u,U}^u\left\|\sum_ja_j\right\|_{L^u(U)}^u+
 \beta_{u,V}^u\left\|\sum_jb_j\right\|_{L^u(V)}^u\\
 &\le(\beta_{u,U}^u+\beta_{u,V}^u)
       \left\|\sum_jd_j\right\|_{L^u(U\oplus_\infty V)}^u,
\end{align*}
which proves \eqref{eq:sumconstant}.
\end{proof}

Now, we are ready to prove Theorem~\ref{thm:main}.
\begin{proof}[Proof of Theorem~\ref{thm:main}]
Fix $1<u<\infty$.

Set
\[
 h_2=\norm{D_2}_{E\to E},\qquad
 K_{p,u}=\beta_{u,E(L^p(\M,\tr))},\qquad
 K_{q,u}=\beta_{u,E(L^q(\M,\tr))}.
\]
Both $K_{p,u}$ and $K_{q,u}$ are finite by \eqref{eq:known-UMD}.
Since $Z$ embeds isometrically as a closed subspace of the direct
sum in Lemma~\ref{lem:Zumd}, Lemma~\ref{lem:direct-sum} implies
\[
 \beta_{u,Z}\le\bigl(K_{p,u}^{\,u}+K_{q,u}^{\,u}\bigr)^{1/u}.
\]
To apply Lemma~\ref{lem:transfer}, define
\[
 C_Q:=h_2(A+B),\qquad L:=2^{1/p}h_2.
\]
Proposition~\ref{prop:Q} yields $\norm{Q}\le C_Q$.
Proposition~\ref{prop:lift} yields, for every $x\in E(\M,\tau)$,
a preimage $z\in Z$ with $Qz=x$ and
$\norm{z}_Z\le L\norm{x}_{E(\M,\tau)}$.
\rev{Applying Lemma~\ref{lem:transfer} with $X=E(\M,\tau)$ gives}
\begin{equation}\label{eq:finalbeta}
 \beta_{u,E(\M,\tr)}
 \le 2^{1/p}h_2^2(A+B)
       \bigl(K_{p,u}^{\,u}+K_{q,u}^{\,u}\bigr)^{1/u}<\infty.
\end{equation}
In particular, every finite $E(\M,\tr)$-valued martingale difference
sequence $(d_j)$ on an arbitrary probability space satisfies
\[
 \left\|\sum_j\varepsilon_jd_j\right\|_{L^u(\Omega;E(\M,\tr))}
 \le C_{E,u}
       \left\|\sum_jd_j\right\|_{L^u(\Omega;E(\M,\tr))}.
\]
By \eqref{eq:known-UMD}, we may take
\[
 C_{E,u}=2^{1/p}h_2^2(A+B)c_{E,u}
            \bigl(b_{p,u}^{\,u}+b_{q,u}^{\,u}\bigr)^{1/u}.
\]
This constant depends only on $E$ and $u$.
\end{proof}


\section{Applications}
\label{sec:extensions}

We deduce the sequence-space and finite-trace versions of
Theorem~\ref{thm:main} by extending the underlying symmetric spaces
to $E(0,\infty)$ while preserving the UMD property.

\subsection{The case of symmetric sequence spaces}

A symmetric Banach sequence space $\ell_E$ is a nonzero Banach space of
complex sequences with the following property: if $b\in \ell_E$ and
$a^*\le b^*$, then $a\in \ell_E$ and $\norm{a}_{\ell_E}\le\norm{b}_{\ell_E}$.
Here $a^*$ is the decreasing rearrangement of $(|a_k|)_{k\ge1}$
with respect to counting measure. We write $\mathbf e_k$ for the
coordinate vectors and $c_0$ for the space of sequences converging
to zero. For a scalar sequence $a$, its associated step function is
\[
 \widetilde a(t)=\sum_{k\ge1}a_k\ind_{[k-1,k)}(t),\qquad t>0.
\]
For background on symmetric spaces and their sequence-space
counterparts, see \cite[Chapter~II, \S4.1, p.~90, and \S8.1,
pp.~157--158]{KPS}.

We use Montgomery-Smith's estimate for independent random variables
\cite[Theorem~1]{MontgomerySmith2002}.
Let $\Omega=(0,1)^{\mathbb N}$ be equipped with the product Lebesgue
probability measure $\Prob$. 
A point $\omega\in\Omega$ is a
sequence $\omega=(\omega_k)_{k\ge1}$, where $\omega_k\in(0,1)$ is its
$k$th coordinate.
The random variables
\[
 U_k(\omega):=k-1+\omega_k,\qquad k\ge1,
\]
are independent because they depend on distinct coordinate
functions, and $U_k$ is uniformly distributed on $(k-1,k)$ because it
is the translate by $k-1$ of the uniformly distributed coordinate
$\omega_k$.

For a measurable function $f:(0,\infty)\to\mathbb C$, set
\begin{equation}\label{eq:uniform-sampling}
 Jf(\omega):=\bigl(f(U_k(\omega))\bigr)_{k\ge1},\qquad
\omega\in\Omega.
\end{equation}
For a random variable $X\in L^1(\Omega,\Prob)$, we write
$\E X=\int_\Omega X\,\mathrm d\Prob$ for its expectation.

\begin{lemma}\label{lem:sequence-extension}
Let $\ell_E$ be a UMD symmetric Banach sequence space on $\mathbb N$.
There is a UMD symmetric Banach function space $F$ on $(0,\infty)$
such that, for every scalar sequence $a=(a_k)_{k\ge1}$,
\[
 a\in \ell_E\quad\Longleftrightarrow\quad
 \widetilde a:=\sum_{k\ge1}a_k\ind_{[k-1,k)}\in F.
\]
Moreover, there are absolute constants $c,C>0$ such that
\begin{equation}\label{eq:sequence-extension-norm}
 c\norm{a}_{\ell_E}\le\norm{\widetilde a}_F\le C\norm{a}_{\ell_E},
 \qquad a\in \ell_E.
\end{equation}
\end{lemma}

\begin{proof}
Without loss of generality, we may assume that
$\norm{\mathbf e_k}_{\ell_E}=1$ for every $k\ge1$.

Since $\ell_E$ is reflexive, it has the Fatou property and an
order-continuous norm; see \cite[Chapter~I, \S2, p.~8, and
Chapter~II, Introduction, pp.~44--45]{KPS}.

Let $F$ consist of the measurable functions $f:(0,\infty)\to\mathbb C$
such that
\[
 f^*\ind_{(0,1)}\in L^2(0,\infty),\qquad
 \left(\int_{k-1}^k f^*(t)\,\mathrm dt\right)_{k\ge1}\in\ell_E,
\]
and equip $F$ with
\begin{equation}\label{eq:sequence-function-norm}
 \norm{f}_F
 :=\norm{f^*\ind_{(0,1)}}_{L^2(0,\infty)}
   +\left\|\left(\int_{k-1}^k f^*(t)\,\mathrm dt\right)_{k\ge1}
     \right\|_{\ell_E}.
\end{equation}
The finite-interval version of this norm is given in
\cite[Section~2, preceding Lemma~4]{MontgomerySmith2002}.
Its triangle inequality extends to $(0,\infty)$ by truncating
functions to $(0,n)$ and using the Fatou property of $\ell_E$.
The defining formula also gives symmetry and the Fatou property.
For every measurable set $A$ of finite measure, $\ind_A\in F$ and
\begin{align}\label{local integral estimate}
 \int_A|f(t)|\,\mathrm dt
 \le\int_0^{|A|}f^*(t)\,\mathrm dt
 \le\lceil|A|\rceil\int_0^1f^*(t)\,\mathrm dt
 \stackrel{\tiny \mbox{\cite[Theorem 4.4.6]{DPS}}}{\le}\lceil|A|\rceil\norm{f}_F.
\end{align}
Here,  $\lceil t\rceil$ denotes the least integer greater than or equal
to $t$, and $|A|$ denotes the Lebesgue measure of $A$.
To verify completeness, suppose that $\sum_j\norm{v_j}_F<\infty$.
The Fatou property and \eqref{local integral estimate} imply that
$\sum_j|v_j|\in F$, so $v:=\sum_jv_j$ exists almost everywhere and
\[
 \left\|v-\sum_{j=1}^n v_j\right\|_F
 \le\sum_{j>n}\norm{v_j}_F\longrightarrow0.
\]
Thus $F$ is a symmetric Banach function space on $(0,\infty)$.

We next show that $J$, as defined in
\eqref{eq:uniform-sampling}, is an isomorphic embedding of $F$ into
$L^2(\Omega;\ell_E)$. For $f\in F$, put
\[
 f_n=f\ind_{(0,n)},\qquad
 J_nf=\bigl(f(U_1),\ldots,f(U_n),0,\ldots\bigr).
\]
The variables $|f(U_k)|$, $1\le k\le n$, are independent and satisfy
\[
 \sum_{k=1}^n\Prob\{|f(U_k)|>s\}
 =\sum_{k=1}^n\int_{k-1}^k\ind_{\{|f|>s\}}(t)\,\mathrm dt
 =d_{f_n}(s),\qquad s>0.
\]
Here and below, for nonnegative quantities $X$ and $Y$, the
notation $X\lesssim Y$ means $X\le C Y$ for a constant $C$ independent
of the quantities under consideration, and $X\asymp Y$ means both
$X\lesssim Y$ and $Y\lesssim X$.
Montgomery-Smith's theorem \cite[Theorem~1]{MontgomerySmith2002}
therefore applies with $M=L^2(0,1)$, $N=\ell_E$ and $Y=f_n^*$
in its notation, and gives
\begin{equation}\label{eq:uniform-finite-estimate}
 \norm{J_nf}_{L^2(\Omega;\ell_E)}
 \asymp
 \norm{f_n^*\ind_{(0,1)}}_{L^2(0,\infty)}
 +\left\|\bigl(f_n^*(k)\bigr)_{k\ge1}\right\|_{\ell_E}
 \asymp\norm{f_n}_F.
\end{equation}
All constants are absolute; in particular, they are independent of
$n$, $f$ and $\ell_E$.

Since $|f_n|\uparrow|f|$, the Fatou property gives
$\norm{f_n}_F\uparrow\norm{f}_F$. Moreover,
\[
 \E\sup_{n\ge1}\norm{J_nf(\omega)}_{\ell_E}^2
 =\sup_{n\ge1}\norm{J_nf}_{L^2(\Omega;\ell_E)}^2
 \lesssim\norm{f}_F^2.
\]
The Fatou property of $\ell_E$ therefore gives
$Jf(\omega)\in\ell_E$ almost surely. Order-continuity of $\norm{\cdot}_{\ell_E}$  implies
$J_nf(\omega)\to Jf(\omega)$ in $\ell_E$ almost surely, so $Jf$ is
strongly measurable. Monotone convergence now yields
\[
 \norm{Jf}_{L^2(\Omega;\ell_E)}^2
 =\lim_{n\to\infty}\norm{J_nf}_{L^2(\Omega;\ell_E)}^2.
\]
Passing to the limit in \eqref{eq:uniform-finite-estimate}, we obtain
\begin{equation}\label{eq:uniform-embedding}
 \norm{Jf}_{L^2(\Omega;\ell_E)}\asymp\norm{f}_F,\qquad f\in F.
\end{equation}
Thus the linear map $J$ has closed range. Since $\ell_E$ is UMD,
$L^2(\Omega;\ell_E)$ is UMD by
\cite[Proposition~1.2(2)]{Bukhvalov}; hence $F$ is UMD by
\cite[Proposition~1.2(1)]{Bukhvalov} and
\eqref{eq:uniform-embedding}.

Finally, for every scalar sequence $a$, fix $k\ge1$ and
$\omega\in\Omega$. Since $U_k(\omega)=k-1+\omega_k\in(k-1,k)$, the
inner sum below, which is taken over $j$, has exactly one nonzero term,
namely the term $j=k$. Thus
\[
\begin{aligned}
J\widetilde a(\omega)
=\bigl(\widetilde a(U_k(\omega))\bigr)_{k\ge1}=\left(
  \sum_{j\ge1}a_j\,
  \mathbf1_{[j-1,j)}(k-1+\omega_k)
  \right)_{k\ge1}
=(a_k)_{k\ge1}=a,
\qquad \omega\in\Omega.
\end{aligned}
\]
If $\widetilde a\in F$, then
\eqref{eq:uniform-embedding} gives
$J\widetilde a\in L^2(\Omega;\ell_E)$, so
$J\widetilde a(\omega)\in\ell_E$ for almost every $\omega\in\Omega$.
The preceding identity shows that $J\widetilde a(\omega)=a$ for every
$\omega\in\Omega$; hence $a\in\ell_E$.
Conversely, if $a\in\ell_E$, then $(\widetilde a)^*=\widetilde{a^*}$.  In particular,  
\[
(\widetilde a)^*(t)=a_k^*
\quad\text{for almost every }t\in[k-1,k),\qquad k\ge1.
\]
Thus, $\widetilde a\in F$ with  
\[
\qquad \norm{a}_{\ell_E}\le 
\norm{\widetilde a}_F
\stackrel{\eqref{eq:sequence-function-norm}}{=}\left(\int_0^1(a_1^*)^2\,\mathrm dt\right)^{1/2} 
 +\norm{a^*}_{\ell_E}  
=\norm{a}_{\ell_\infty}+\norm{a}_{\ell_E} \le2\norm{a}_{\ell_E}.
\]
The proof is complete. 
\end{proof}

Theorem~\ref{thm:main} and Lemma~\ref{lem:sequence-extension} yield
the following corollary, which answers the question of Lust-Piquard
and Pisier \cite[p.~251]{LP}.

\begin{corollary}\label{cor:sequence-ideal-umd}
Let $\ell_E$ be a UMD symmetric Banach sequence space. Then
\[
 \begin{gathered}
 C_E=\{x\in\mathcal K(\ell^2):(s_n(x))_{n\ge1}\in \ell_E\},
 \\
 \norm{x}_{C_E}=\norm{(s_n(x))_{n\ge1}}_{\ell_E},
 \end{gathered}
\]
is UMD. Here $\mathcal K(\ell^2)$ is the space of compact operators
on $\ell^2$, and $s_n(x)$ are the singular values of $x$, in
non-increasing order and counted with multiplicity.
\end{corollary}

\subsection{The case of function spaces on a finite interval}
\label{subsec:finite-interval}

Recall that every UMD symmetric Banach function space $E$ on $(0,1)$
has nontrivial Boyd indices, that is, $1<p_E\le q_E<\infty$;
see Theorem~\ref{thm:type-cotype} and \cite[p.~132]{LT2}.

We use the extension of a symmetric function space on $(0,1)$
constructed by Johnson, Maurey, Schechtman and Tzafriri
\cite[Section~8]{JMST}; see
\cite[Theorem~2.f.1(ii), p.~203]{LT2}
(see also \cite{AS2010,Astashkin2011}).

\begin{lemma}\label{lem:finite-interval-extension}
Let $E$ be a UMD symmetric Banach function space on $(0,1)$.
Set
 $F=Z_E^2$, consisting of all measurable functions
$g$ on $(0,\infty)$ for which the following norm is finite:
\begin{equation}\label{eq:finite-extension-norm}
 \norm{g}_F=
 \max\left\{
   \norm{g^*\ind_{(0,1)}}_E,
   \left(\sum_{n=0}^{\infty}
       \left(\int_n^{n+1}g^*(t)\dd t\right)^2
       \right)^{1/2}
 \right\},
\end{equation}
which is equivalent to
\[
 \norm{g^*\ind_{(0,1)}}_E
 +\norm{g^*\ind_{(1,\infty)}}_{L^2(0,\infty)}.
\]
The space $E$ is isomorphic to $Z_E^2$. 
\end{lemma}




For a symmetric Banach function space $E$ on $(0,1)$ and a
von Neumann algebra $\M$ with a faithful normal tracial state
$\tau$, we use the notation\cite{LSZ,KS}
\[
 E(\M,\tau)=\{x\in S(\M,\tau):\mu(x)|_{(0,1)}\in E\},
 \qquad
 \norm{x}_{E(\M,\tau)}=\norm{\mu(x)|_{(0,1)}}_E.
\]

Let $F=Z_E^2$ be the space from
Lemma~\ref{lem:finite-interval-extension}; since $F$ is isomorphic to
$E$, it is a UMD space. If $\tau$ is a tracial state, then $\tau(\ind)=1$ and
$\mu(x)$ vanishes on $[1,\infty)$ for every $x\in S(\M,\tau)$; hence
 $E(\M,\tau)$ and $F(\M,\tau)$ coincide. 
Applying
Theorem~\ref{thm:main} to $F$ gives the following result.
\begin{corollary}\label{cor:finite-trace}
Let $E$ be a UMD symmetric Banach function space on $(0,1)$.
For every finite von Neumann algebra $\M$ equipped with a faithful
normal tracial state $\tau$, the space $E(\M,\tau)$ is UMD.
Moreover, for every $1<u<\infty$ there is a constant
$C_{E,u}$ such that
\[
 \beta_{u,E(\M,\tau)}\le C_{E,u}
\]
for all such $(\M,\tau)$.
\end{corollary}


\end{document}